\documentclass[12pt]{amsart}
\usepackage{color,graphicx,array, amssymb, amscd,slashed,fancyhdr}
\newtheorem{Theorem}{Theorem}[section]
\newtheorem{Lemma}[Theorem]{Lemma}

\newtheorem{Corollary}[Theorem]{Corollary}
\theoremstyle{definition}
\newtheorem{Definition}{Definition}
\theoremstyle{remark}

\newtheorem{Remark}[Theorem]{Remark} 
\numberwithin{equation}{section}
\newcommand{\R}{\mathbb R}
\newcommand{\C}{\mathbb C}

\newcommand{\D}{\mathbb D}

\newcommand{\ISU}{{\rm SU}_{1, 1}}
\newcommand{\Uone}{{\rm U}_1}

\newcommand{\isu}{\mathfrak{su}_{1, 1}}

\newcommand{\LISU}{(\Lambda {\rm SU}_{1, 1})_\sigma}

\newcommand{\ad}{\operatorname{Ad}}
\newcommand{\di}{\operatorname{diag}}
\newcommand{\tr}{\operatorname{Tr}}
\newcommand{\Nil}{{\rm Nil}_3}

\renewcommand{\Re}{\operatorname {Re}}
\renewcommand{\Im}{\operatorname {Im}}

\newcommand{\be}{\begin{equation*}}
\newcommand{\ee}{\end{equation*}}
\newcommand{\Lt}{\mathbb L_3}
\renewcommand{\l}{\lambda}

\begin{document}
\title{A solution to the  Bernstein problem in the 
 three-dimensional Heisenberg group via loop groups}
 \author[J. F.~Dorfmeister]{Josef F. Dorfmeister}
 \address{Fakult\"at f\"ur Mathematik, 
 TU-M\"unchen, Boltzmann str. 3, D-85747,  Garching,  Germany}
 \email{dorfm@ma.tum.de}
\author[J.~Inoguchi]{Jun-ichi Inoguchi}
 \address{Department of Mathematical Sciences, 
 Faculty of Science, Yamagata University,  Yamagata, 990--8560, Japan}
 \email{inoguchi@sci.kj.yamagata-u.ac.jp}
 \thanks{The second named author is partially supported by Kakenhi 24540063}
\author[S.-P.~Kobayashi]{Shimpei Kobayashi}
 \address{Department of Mathematics, Hokkaido University, Sapporo, 060-0810, Japan}
 \email{shimpei@math.sci.hokudai.ac.jp}
 \thanks{The third named author is partially supported by Kakenhi 26400059}
\subjclass[2010]{Primary~53A10, Secondary~53C42}
\keywords{Bernstein problem; minimal graphs; Heisenberg group; 
 loop groups; spinors}
\date{\today}
\begin{abstract} 
 In this note we present a short alternative proof 
 for the Bernstein problem in the three-dimensional Heisenberg group $\Nil$ 
 by using the loop group technique.
\end{abstract}
\maketitle
\section*{Introduction}
 The {\it Bernstein problem} is one of the traditional problems 
 of global differential geometry.
 The original result due to Bernstein asserts that 
 every entire minimal graph in Euclidean three-space 
 $\mathbb{R}^3(x_1,x_2,x_3)$ is a plane.
 In other words, 
 Bernstein's result shows that the only global solution 
 on the $(x_1, x_2)$-plane to the
 so-called {\it minimal surface equation} 
\be
\{1+(f_{x_1})^2\}f_{x_2x_2}-2f_{x_1}f_{x_2}f_{x_1x_2}
+\{1+(f_{x_2})^2\}f_{x_1x_1}=0  
\ee
 is a linear function of $x_1$ and $x_2$.

 The Bernstein problem has been generalized  
 to  a problem basically asking for a classification of all 
 entire minimal graphs.
 On the other hand, when the ambient space is not the Euclidean 
 three-space, the Bernstein problem often needs to be amended.
 For instance, in Minkowski three-space $\mathbb{L}_3$ equipped with the
 natural Lorentz metric $dx_1^2+dx_2^2-dx_3^2$,  
 there are many entire (timelike) minimal graphs over the timelike plane 
 $\mathbb{L}_2=\mathbb{R}^2(x_2,x_3)$, see for example \cite{Mil1}. 

 Next, we focus on the three-dimensional \textit{Heisenberg group} $\Nil$ 
 which is one of the model spaces of Thurston geometries \cite{Thurston}.
 The space $\Nil$ is realized as Cartesian three-space 
 $\mathbb{R}^3(x_1,x_2,x_3)$ equipped with the Riemannian metric
\be
 dx_1^2+dx_2^2+\left\{dx_3+\frac{1}{2}(x_2dx_1-x_1dx_2)\right\}^2
\ee
 and a nilpotent Lie group structure, see for example \cite{DIK:mini}. 
 The Riemannian metric is invariant under 
 the nilpotent Lie group structure and has a 4-dimensional isometry group.
 The identity component of the isometry group is a semi-direct 
 product $\mathrm{Nil}_3\ltimes \mathrm{SO}_2$.

 It has been known for a long time that in 
 $\mathrm{Nil}_3$ nontrivial entire minimal graphs exist, see for example
 \cite{IKOS}. Therefore, in $\Nil$ the 
 Bernstein problem has been phrased more specifically 
 as the problem to construct 
 entire minimal graphs over the natural $(x_1, x_2)$-plane with a
 prescribed holomorphic quadratic differential.
 
 Under this formulation, Fern{\'a}ndez and Mira studied the Bernstein problem in 
 $\Nil$ \cite{Fer-Mira2}. They proved that for a prescribed 
 holomorphic quadratic differential $Q\,dz^2$ over the complex plane 
 $\C$ with $Q\neq 0$ or the unit disc $\mathbb{D}$,
 there exists a two-parameter family of entire 
 minimal vertical graphs whose Abresch-Rosenberg differential is $Q\,dz^2$.
 Their proof relies firstly on the Lawson-type correspondence 
 (often called \textit{sister correspondence}) between 
 minimal surfaces in $\mathrm{Nil}_3$ and 
 surfaces of constant mean curvature ({\sc CMC} in short) 
 with mean curvature $H =1/2$ in 
 the product space $\mathbb{H}^2\times\mathbb{R}$, 
 where $\mathbb H^2$ denotes the hyperbolic two-space. 
 Secondly, they use the correspondence between harmonic maps into
 $\mathbb{H}^2$ and {\sc CMC} surfaces with mean curvature $H =1/2$ 
 in the product space $\mathbb{H}^2\times\R$. 
 Finally, they use a result of Wan and Au \cite{Wan, WanAu}
 solving the Bernstein problem for spacelike {\sc CMC} surfaces in 
 $\mathbb{L}_3$ and use that  the Gauss map of 
 those surfaces is also harmonic into $\mathbb{H}^2$.

 Our proof is much more direct.
 In our previous work \cite{DIK:mini}, 
 we have established a generalized Weierstrass 
 type representation for 
 minimal surfaces in $\mathrm{Nil}_3$.
 Every simply connected (nowhere vertical) minimal surface
 is obtained from an extended frame for a harmonic map 
 into the hyperbolic two-space $\mathbb{H}^2$. 
 In this paper we give a short proof of
 the solution to the Bernstein problem in $\mathrm{Nil}_3$
 by virtue of the generalized Weierstrass 
 type representation established in \cite{DIK:mini}.
 The advantage of our approach is that we can give a direct 
 relation between minimal graphs in $\mathrm{Nil}_3$ and 
 spacelike {\sc CMC} surface with mean curvature $H=1/2$ 
 graphs in $\mathbb{L}_3$, Theorem \ref{thm:symcorrespond}.
 This relation enables us to give a simple alternative proof 
 of Fern{\'a}ndez-Mira's theorem, Theorem \ref{thm:Bern}.

 Our new proof actually also provides new insights.
 In fact our proof clarifies the geometric meaning of 
 the two-parameter ambiguity of entire minimal graphs with 
 prescribed Abresch-Rosenberg differential. While it is quite 
 clear that the two-parameter family is related to the boosts 
 in $\ISU$, our argument also shows how the corresponding family of surfaces varies in $\Nil$.

\section{Bernstein problem}
 We discuss the Bernstein problem in $\Nil$, that is, 
 the classification of entire minimal vertical graphs in 
 $\Nil$. We only consider vertical graphs. Therefore we will sometimes omit the word 
 ``vertical''.  In Appendix \ref{sc:Preliminaries}, we give 
 a short review of facts and results of \cite{DIK:mini} which are used 
 for the solution to the Bernstein problem via loop groups.
 From now on, we denote the coordinates of $\Nil$ or $\Lt$ by $(x_1, x_2, x_3)$.

\subsection{Completeness}\label{sc:Com}
 The basic result used in this paper is Theorem \ref{thm:Sym}. It explains the direct relation between 
 minimal surfaces in $\Nil$ and spacelike {\sc CMC} surfaces in $\Lt$.
 This close relationship is also underlined by a simple relation between the corresponding metrics.

\begin{Lemma}\label{lm:metricrelation}
 Let $f^{\l}$ and $f_{\Lt}^{\l}$ be an associated  family of minimal surfaces in $\Nil$ 
 and an associated family of spacelike {\sc CMC} surfaces with mean curvature $H =1/2$
 in $\Lt$  correlated and defined as in Theorem 
 \ref{thm:Sym}, respectively. Denote the metric of $f^{\l}$ by $e^u dz d\bar z$
 and the metric of $f_{\Lt}^{\l}$ by  $e^{u_{\Lt}}dz d \bar z$, respectively.
 Moreover, let $\phi_3^{\l}dz$ be the coefficient of $e_3$ in 
 $(f^{\l})^{-1} f_{z}^{\l} dz
 = \sum_{i =1 }^3 (\phi_{i}^{\l} dz) e_{i}$. Then the following relation holds$:$
\be
e^{u_{\Lt}} +  4 |\phi_3^{\l}|^2 = e^{u}.
\ee
\end{Lemma}
\begin{proof}
 It is known that the conformal factors $e^{u}$ and $e^{u_{\Lt}}$
 can be computed explicitly 
 in terms of spinors, see \cite[Section 3.1]{DIK:mini},
 Remark \ref{rm:AppRem} and  \cite{BRS:Min}:
 \be 
e^u =  4 (|\psi_1^{\l}|^2 + |\psi_2^{\l}|^2)^2, \;\;\; e^{u_{\Lt}}=  4 (|\psi_1^{\l}|^2 - |\psi_2^{\l}|^2)^2,
\ee
 where $\psi_j^{\l} \;\;(j =1, 2)$ is a family of spinors for 
 the associated family $f^{\l}$.
 Since $\phi_3^{\l} = 2 \psi_1^{\l} \overline{\psi_2^{\l}}$, the claim follows.
\end{proof}

\begin{Remark}\label{rm:metrics}
 It is known that the metrics $e^{u_{\Lt}} dz d \bar z$ of an associated family 
 of spacelike {\sc CMC} surfaces $f_{\Lt}^{\l}$
 are independent of $\l$, that is, on a simply connected domain, any 
  two members of the associated family $\{f_{\Lt}^{\l}\}_{\l \in \mathbb S^1}$ are isometric. In fact the metric can be computed by the support $h (dz)^{1/2} (d \bar z)^{1/2}$, 
 see Appendix \ref{sc:Preliminaries} for definition, as
 \begin{equation}\label{eq:supportmetric}
 h^2 dz d \bar z= e^{u_{\Lt}}dz d \bar z.
 \end{equation}
 However, the metrics $e^{u} dz d \bar z$  of an associated family of 
 minimal surface $f^{\l}$ depend on $\l$, that is, 
 any two members of the associated family $\{f^{\l}\}_{\l \in \mathbb S^1}$ 
 are,  in general, non-isometric. 
\end{Remark}
 Using the relation above, we have the following theorem.

\begin{Theorem}\label{thm:compl}
 Let $f^{\l}$ and $f_{\Lt}^{\l}$ be an associated  family of minimal surfaces in $\Nil$ 
 and an associated family of spacelike {\sc CMC} surfaces with mean curvature $H =1/2$
 in $\Lt$ correlated and defined as in Theorem 
 \ref{thm:Sym}, respectively. Assume that one member of the associated family 
 $\{f_{\Lt}^{\l}\}_{\l \in \mathbb S^1}$ is closed with respect to 
 the Euclidean topology. Then each member of the associated family $\{f^{\l}\}_{\l \in \mathbb S^1}$ is a complete, entire graph.
\end{Theorem}
\begin{proof}
 We denote the spacelike {\sc CMC} surface in $\Lt$ which is closed with respect to the 
 Euclidean topology by $f_{\Lt}^* = f_{\Lt}^{\l}|_{\l^* \in \mathbb S^1}$.
 From the assumption and by \cite[p.~415]{CY}, we conclude 
 that $f^{*}_{\Lt}$ is complete. Moreover from \cite[Proposition 2]{Wan}, 
 $f^{*}_{\Lt}$ is also an entire graph. 
 Since within the associated family $\{f_{\Lt}^{\l}\}_{\l \in \mathbb S^1}$ 
 the metric is invariant, see Remark \ref{rm:metrics}, 
 each member of the associated family of spacelike {\sc CMC} surfaces 
 $\{f_{\Lt}^{\l}\}_{\l \in \mathbb S^1}$ is also complete.
 Then from Lemma \ref{lm:metricrelation}, we have that 
 each member of the associated family of minimal surfaces 
 $\{f^{\l}\}_{\l \in \mathbb S^1}$ is complete.
 
 Let us  look more  closely at the correspondence between 
 $f_{\Lt}^{\l}$ and $f^{\l}$.
 From  formulas  \eqref{eq:SymMin} and \eqref{eq:symNil} we infer by inspection 
 that $f^{\l}$ and $f^{\l}_{\Lt}$ share the same $x_1$-, $x_2$-components.
 Here, as pointed out before, we denote the coordinates 
 of $\Nil$ or $\Lt$ by $(x_1, x_2, x_3)$.
 Therefore, since $f_{\Lt}^{\l}$ is an entire graph, thus $f^{\l}$
 also is an entire graph. This completes the proof.
\end{proof}

\subsection{Rigid motions}\label{sc:Rigid}
 It is known that the isometry group of $\Lt$ is the six-dimensional Lie group
 which is generated by 
 a one-parameter family of rotations around $x_3$-axis (the timelike axis), 
 a two-parameter family of boosts and three families of translations.
 In contrast, the isometry group of $\Nil$ is only four-dimensional and 
 is generated by a one-parameter family of rotations around $x_3$-axis 
 and three families of translations  in $\Nil$.

 A comparison of the two Sym formulas in Theorem \ref{thm:Sym} 
 indicates that isometries of Minkowski space will not necessarily 
 become isometries of $\Nil$. The precise relation will be made 
 clear in the Lemma below.
\begin{Lemma}\label{lem:isometry}
 Let $f_{\Lt}^{\l}$ and $\tilde{f}_{\Lt}^{\l}$ be two associated family of 
 spacelike {\sc CMC} surfaces 
 with mean curvature $H=1/2$ in $\Lt$ defined by the Sym-formula 
 in Theorem  \ref{thm:Sym} for 
 some extended frames $F^{\l}$ and $\tilde F^{\l}$, respectively and set 
 $f_{\Lt}= f_{\Lt}^{\l}|_{\l=1}$ and $\tilde f_{\Lt}=\tilde f_{\Lt}^{\l}|_{\l=1}$.
 Moreover, let $f^{\l}$ and $\tilde f^{\l}$ denote the two associated families of
 minimal surfaces in 
 $\Nil$ \ defined from the same extended frames $F^{\l}$ and $\tilde F^{\l}$, 
 respectively and set  $f= f^{\l}|_{\l=1}$ and $\tilde f=\tilde f^{\l}|_{\l=1}$.   
 Assume that $f_{\Lt}$ and $\tilde f_{\Lt}$
 are isometric by some rigid motion in $\Lt$.
 Then the following statements hold$:$
\begin{enumerate}
\item If $f_{\Lt}$ and $\tilde f_{\Lt}$
 are isometric by a rotation around $x_3$-axis $($the timelike axis$)$, then 
 $f$ and $\tilde f$ are isometric by the rotation around 
 $x_3$-axis $($the same angle$)$ and some translation.
\item If $f_{\Lt}$ and $\tilde f_{\Lt}$
 are isometric by a translation, then 
 $f$ and $\tilde f$ are isometric by 
 some translation $($not necessarily the same translation$)$.
\item If $f_{\Lt}$ and $\tilde f_{\Lt}$
 are isometric by a boost, then 
 $f$ and $\tilde f$ are, in general,  not isometric.
\end{enumerate}
\end{Lemma} 
\begin{proof}
 Since $f_{\Lt}(= f_{\Lt}^{\l}|_{\l=1})$ and 
 $\tilde f_{\Lt}(= \tilde f_{\Lt}^{\l}|_{\l=1})$
 are isometric by a rigid motion in $\Lt$, the isometry between 
 these two surfaces lifts to the level of frames 
$F= F^{\l}|_{\l=1}$ 
 and $\tilde F = \tilde F^{\l} |_{\l=1}$ 
 as $\tilde F = MFk$, where $M$ is a $z$-independent $\ISU$-valued matrix and 
 $k$ is a $\Uone$-valued matrix. 
 After introducing the loop parameter we obtain the relation
\begin{equation}\label{eq:isoF}
 \tilde F^\lambda = M^\lambda F^\lambda k.
\end{equation}
 Note that $M^{\l}$ is a $\LISU$-valued matrix and satisfies $M^{\l}|_{\l=1} = M$.
 Then it is easy to see that $\tilde f_{\Lt}^{\l}$ and 
 $f_{\Lt}^{\l}$ satisfy the  relation:
\begin{equation}\label{eq:rigidmotion}
 \tilde f_{\Lt}^{\l} = M^\lambda f_{\Lt}^{\l} (M^\lambda)^{-1} -i \l (\partial_{\l} M^\lambda) (M^\lambda)^{-1}.
\end{equation}
 Now a straightforward computation shows that 
 the corresponding two minimal surfaces $f^{\l}$ and $\tilde f^{\l}$  
 have the following relation:
\begin{equation}\label{eq:minirelation}
 \tilde f^{\l} = \left(\ad (M^\lambda) (f^{\l}_{\Lt}) - X^{\l}\right)^{o}
 - \frac{1}{2} \lbrace \ad(M^\lambda) (i \l \partial_{\l} f^{\l}_{\Lt})
 + [X^{\l}, \ad (M^\lambda)(f^{\l}_{\Lt})] - Y^{\l}\rbrace^{d},
\end{equation}
 where we set 
\begin{equation}\label{eq:XY}
 X^{\l} = i \l (\partial_{\l} M^\lambda)(M^\lambda)^{-1}, \;\;
 Y^{\l} = i \l (\partial_{\l} X^{\l}),
\end{equation} 
 and the superscripts ``$o$'' and ``$d$'' 
 denote the off-diagonal and  diagonal part, 
 respectively. 
For simplicity of notation we do not distinguish here
 $f^{\l}$ in $\Nil$ and  $\hat f^{\l}$ in $\isu$.

 We note that for each fixed $\l \in \mathbb S^1$, 
 the first part of the right-hand side in \eqref{eq:rigidmotion}
 describes a Lorentz transformation and 
 the second part of the right-hand side in \eqref{eq:rigidmotion}
 describes a translation, respectively.  
 We now consider each of the three types of generators separately:

 $(1)$ First suppose that $\tilde f_{\Lt}$ and $f_{\Lt}$
 are isometric by a rotation around $x_3$-axis (the timelike axis).
 Since the original transformation $M$ was a rotation, it follows 
\be
M (=  M^\lambda|_{\l =1}) = {\rm diag }(e^{i \theta}, e^{-i \theta}),
\;\;\;\;
\partial_{\l} M^\lambda|_{\l =1} = \mathbf 0.
\ee
 Here $2 \theta$  is the angle of rotation. 
 A straightforward computation shows that 
 $f$ and $\tilde f$ satisfy the equation 
 $\tilde{f} = \ad(M) (f) + \frac{1}{2}Y^d$,
 where the translation term $Y$ can be computed as
\be
Y = Y^{\l}|_{\l =1} = - \l^2 (\partial_{\l}^2 M^\lambda) (M^\lambda)^{-1}|_{\l =1}.
\ee
 Therefore, this one-parameter family consists of isometric minimal surfaces 
 in $\Nil$.

 $(2)$ Next suppose  that $f_{\Lt}$ and $\tilde f_{\Lt}$
 are isometric by some translation. 
 Since the original transformation $M$  was a translation, it follows
\be
 M(= M^\lambda|_{\l =1} )= {\rm id}, \;\;(\partial_{\l} M^\lambda) |_{\l =1}\neq \mathbf 0.
\ee 
 Substituting $\lambda = 1$ into   \eqref{eq:minirelation}  we see immediately  
 that $f$ and $\tilde f$ 
 satisfy the relation $\tilde{f} = f + A$,
 where $A = A(x_1,x_2)$ is given by 
 \begin{equation}
 A =  - X^o - \frac{1}{2} \left([X, f_{\Lt}]  -Y \right)^d,
\end{equation}
 where $X= X^{\l}|_{\l=1}$ and $Y= Y^{\l}|_{\l=1}$ for $X^{\l}$ and 
 $Y^{\l}$ in \eqref{eq:XY}.
 It is clear that $Y$ is independent of $x_1$ and $x_2$, the coordinates for 
 $f$, but  $x_1$ and $x_2$ enter the commutator.
 An explicit computation, using the bases stated in Appendix \ref{sc:Sym} 
 and the transformation formula stated in Appendix \ref{sc:isoNil}, 
 now shows that $\tilde{f}$ can be obtained from $f$ by a translation in $\Nil$ 
 (with a constant vector, whose coefficients  basically are the components 
 of $X^{o}$ and of $Y^d$).

 $(3)$ Let us finally consider  the transformations $M$ given by boosts in $\Lt$. 
 These
 transformations form a two-parameter family. 
 Since the original transformation was a boost, 
 it follows  
\begin{equation}\label{eq:blement}
 M = M^\lambda|_{\l =1} = \begin{pmatrix} \alpha  & \beta \\ 
 \bar{\beta} & \alpha \end{pmatrix},
\;\;\;\;
\partial_{\l} M^{\l}|_{\l =1} = \mathbf 0.
\end{equation}
 Here $\alpha \in \R$, $\beta \in \C$ and $\alpha^2 - |\beta|^2=1$
and we obtain
\be
 \tilde f = \left(\ad (M) (f_{\Lt}) \right)^{o}
 - \frac{1}{2} \left( \ad (M) (i \l \partial_{\l} f^{\l}_{\Lt} |_{\lambda = 1})
  - Y \right)^{d},
\ee
 where $Y= Y^{\l}|_{\l=1}$ for $Y^{\l}$ in \eqref{eq:XY}.
 Now it follows by  a straightforward computation that 
\begin{equation}\label{eq:boost}
\tilde f
=\begin{pmatrix}
(\alpha^2 + |\beta|^2) r + \alpha \beta \bar{s} - \alpha \bar{\beta} s & -2 \alpha \beta p 
+ \alpha^2 q - \beta^2 \bar q   \\
2 \alpha \bar{\beta} p + \alpha^2 \bar{q} - \bar{\beta}^2  q   & -(\alpha^2 + |\beta|^2) r 
-\alpha \beta \bar {s} + \alpha \bar{\beta}  s
\end{pmatrix} 
+ \frac{1}{2}Y^{d},
\end{equation}
 where $p,r\in i \R$ and $q, s \in \C$ are functions defined by 
\begin{equation}\label{eq:pqrs}
f_{\Lt}^{\l}|_{\l=1} =\begin{pmatrix}p & q  \\\bar q & -p\end{pmatrix}\;\mbox{and}\;
- \frac{1}{2}i \l (\partial_{\l}f_{\Lt}^{\l})|_{\l=1} 
=\begin{pmatrix}r & s  \\\bar s & -r\end{pmatrix}.
\end{equation}
 Note that the components  of the minimal surface $f$ 
 in the basis of Appendix \ref{sc:Sym} are 
 given by 
\be
(x_1, x_2, x_3) =(2 \Im q, -2 \Re q, -2 \Im r).
\ee
 Thus from \eqref{eq:boost} and the action of the isometry group of $\Nil$ 
 as described in \eqref{eq:action}, it is easy to see that 
 $f$ and $\tilde f$ are in general not isometric, see Remark \ref{rm:conditions}
 in detail.
\end{proof}
\begin{Remark}\label{rm:conditions}
 In case $(3)$ of Lemma \ref{lem:isometry}, 
 from \eqref{eq:boost} and the action of the isometry group of $\Nil$ in \eqref{eq:action},
 we see that the $f$ and $\tilde f$ are isometric in $\Nil$ if and only if 
 there exist some $\theta\in \R, (a_1, a_2, a_3) \in \R^3$
 such that the following two equations hold:
\begin{eqnarray*}
(\alpha^2 + |\beta|^2)x_3  -4\alpha \Im (\beta\bar s)  - Y^{11}= 
 a x_1 + b x_2 + x_3+ a_3,&\\
-4 \alpha \beta p 
+ i (\beta^2 + \alpha^2)x_1
+ (\beta^2 - \alpha^2)x_2 
=e^{i\theta}(i x_1 -x_2) + i a_1 - a_2, 
&
\end{eqnarray*}
 where $Y^{d} = \di (i Y^{11}, -i Y^{11})$, 
 $a = \frac{1}{2} (a_1 \sin \theta - a_2 \cos \theta), 
 b =\frac{1}{2} (a_1 \sin \theta + a_2 \cos \theta)$, and 
 $p, s$ are purely imaginary, complex valued functions, respectively, 
 defined in \eqref{eq:pqrs}. 
 From these two equations,
 it is easy to see that they are satisfied for very special 
 minimal surfaces $f$ only.
\end{Remark}

\begin{Remark}
 After fixing base points, the Sym-formula establishes a 
 $1$-$1$-relation between spacelike {\sc CMC} with mean curvature $H=1/2$ 
 surfaces in $\Lt$ and minimal surfaces in $\Nil$. 
 Clearly, the Poincar\'e group $\ISU \ltimes \Lt$ acts on the family 
 of spacelike {\sc CMC} 
 surfaces in $\Lt$.
 If we fix base points, we eliminate the action of the translation part of 
 the Poincar\'e group, reducing the action to the Lorentz group $\ISU$.

 Via the Sym-formula, the Poincar\'e group also acts on the family of minimal 
 surfaces in $\Nil$.  Since we fix base points, we can also eliminate the translation 
 part of the isometry group of $\Nil$. So generically, the dimension of the 
 family of minimal surfaces should be three.
 But from Lemma \ref{lem:isometry}, identifying 
 minimal surfaces which are isometric 
 by rotations, we see that two is the highest dimension of any orbit. These orbits are realized by the action  of boosts.
From \eqref{eq:blement}, the set of boosts $\mathcal B$ can be computed as 
\be
\mathcal B = \left\{ \left.X \, {}^t\!\bar X \;\right|\; X \in \ISU \right\}. 
\ee
 From this it is clear that $\mathcal B$ is the symmetric space 
 $\ISU/\Uone$.
\end{Remark}

\subsection{Bernstein problem}\label{sc:Bernstein}
 We will finally present a short  alternative proof of 
 the Bernstein problem in $\Nil$ using the loop group method. 
 The heart of the proof is the following simple relation between 
 spacelike {\sc CMC} graphs in $\Lt$ and minimal graphs in $\Nil$.
\begin{Theorem}\label{thm:symcorrespond}
 Every entire, complete, spacelike {\sc CMC} graph  in $\Lt$ 
 with mean curvature $H =1/2$ and the Hopf differential $Q_{\Lt} dz^2$ induces, 
 via the Sym-formula
 $($applied to its associated family$)$, an entire, 
 complete, minimal graph in $\Nil$ with  Abresch-Rosenberg differential
 $-Q_{\Lt} dz^2$.

 Conversely every entire, complete, minimal graph in $\Nil$ is obtained in this way.
\end{Theorem}
\begin{proof}
 Let $g_{\Lt}$ be an  entire complete  
 spacelike {\sc CMC} graph with mean curvature $H=1/2$
 over the $(x_1,x_2)$-plane 
 in $\Lt$ whose Hopf differential   
 is $Q_{\Lt}dz^2$. 
 Let $F^{\l}$ be the  extended frame of $g_{\Lt}$, 
 see Remark \ref{rm:AppRem} for 
 the definition, and 
 apply the Sym-formulas of Theorem \ref{thm:Sym} to obtain  $f^{\l}_{\Lt}$ 
 and $f^{\l}$ from the same extended frame $F^{\l}$. 
 Note that $f^{\l}_{\Lt}$ and $f^{\l}$ define an associated family of 
 spacelike {\sc CMC} surfaces in $\Lt$ and minimal surfaces in $\Nil$, respectively.
Moreover,  $f_{\Lt}= f_{\Lt}^{\l}|_{\l =1}$ and $g_{\Lt}$ are isometric by 
 some rigid motion in $\Lt$, by the fundamental theorem of surface theory 
 (the mean curvature $H$, the Hopf differential $Q_{\Lt} dz^2$ and 
 the metric $e^{u_{\Lt}} dz d \bar z$ are the same),
 thus $f_{\Lt}$ is also a entire, complete,  
 spacelike {\sc CMC} graph, \cite[Proposition 1]{Wan}. 
 From  formulas  \eqref{eq:SymMin} and \eqref{eq:symNil} 
 we infer by inspection that $f^{\l}$ and $f^{\l}_{\Lt}$ share the 
 same $x_1$-, $x_2$-components. Thus $f= f^{\l}|_{\l =1}$ 
 is an entire minimal graph as well. 
 Moreover, from Theorem \ref{thm:compl} we obtain that $f(= f^{\l}|_{\l =1})$ 
 is complete and by Remark \ref{rm:AppRem} 
 we know that  the Abresch-Rosenberg differential is $-Q_{\Lt} dz^2$. 

 To verify the second statement, 
 let $f$ be an entire, complete, minimal graph in $\Nil$ whose 
 Abresch-Rosenberg differential is $Q \,dz^2$ and let 
 $F^{\l}$ be the extended frame of $f$ and $f^\lambda$ its associated family
 from the extended frame $F^{\l}$. 
 Then we have $f=f^{\l}|_{\l =1}$ up to translation in $\Nil$. 
 Moreover, let $f_{\Lt}^{\l}$ 
 be the spacelike {\sc CMC} surface in $\Lt$ defined by the same extended 
 frame $F^{\l}$
 in \eqref{eq:SymMin}. 
 Note that the Hopf differential of $f_{\Lt}^{\l}$ is $
 Q_{\Lt}^{\l} dz^2 = -\l^{-2} Q \,dz^2$.
 Since $f$ and  $f_{\Lt}^{\l}|_{\l =1}$ have the same $x_1$-,$x_2$-components, 
 the latter surface is an entire {\sc CMC} 
 graph in $\Lt$, and thus by \cite[p.~415]{CY} it is complete.
\end{proof}
 Using Theorem \ref{thm:symcorrespond}, 
 it is easy to give the proof of the solution to the  Bernstein problem.
\begin{Theorem}\label{thm:Bern}
 Let  $Q \,dz^2$ be a holomorphic quadratic differential 
 on $\D$ or $M = \C$ with $Q \not\equiv 0$. Then the following statements hold$:$
\begin{enumerate}
\item There exists a two-parameter family of entire, 
 complete, minimal graphs in $\Nil$, whose Abresch-Rosenberg 
 differential is $Q\, dz^2$. 

\item Any two members of this two-parameter family are generically 
  non-congruent.

\item Each member of this two-parameter  family  is induced via the Sym-formula 
 by $($the associated family of$)$ an entire, complete, spacelike {\sc CMC} graph 
 in $\Lt$ with the Hopf differential $-Q \,dz^2$. 
\end{enumerate}
\end{Theorem}
\begin{proof}
 First we note that it  is known that for 
 a given holomorphic quadratic differential 
 $Q_{\Lt}dz^2$ on $\mathbb D$ or $\mathbb C$, 
 there exists a unique entire complete  
 spacelike {\sc CMC} graph $g_{\Lt}$ over the $(x_1,x_2)$-plane 
 in $\Lt$ whose Hopf differential   
 is $Q_{\Lt}dz^2$, \cite{Wan, WanAu}.  
 Here ``unique'' means that any other such spacelike {\sc CMC} 
 graph whose Hopf differential 
 is $Q_{\Lt} dz^2$ is isometric to $g_{\Lt}$
 by an isometry of $\Lt$. We normalize the mean curvature of $g_{\Lt}$ 
 as $H =1/2$ and set $Q_{\Lt} dz^2= -Q \,dz^2$, where $Q \,dz^2$ is the quadratic 
 differential satisfying the condition in the Theorem.
 Let $F^{\l}$ be the  extended frame of $g_{\Lt}$, see Remark \ref{rm:AppRem} for 
 the definition, 
 and apply the Sym formulas of Theorem \ref{thm:Sym} to obtain  $f^{\l}_{\Lt}$ 
 and $f^{\l}$ from the same extended frame $F^{\l}$. 
 From Theorem \ref{thm:symcorrespond} we know that $f_{\Lt}=f^{\l}_{\Lt}|_{\l=1}$ 
 and $f= f^{\l}|_{\l=1}$ are complete entire graphs.
 From the construction it is clear that $f$ is a minimal surface. 
 Moreover the Abresch-Rosenberg differential of $f$ is $Q \, dz^2$.

 We now consider a spacelike {\sc CMC} surface $\tilde g_{\Lt}$ isometric 
 to $g_{\Lt}$ in $\Lt$. Then as explained in the proof of Lemma \ref{lem:isometry},
  the extended frame $\tilde F^{\l}$ of $\tilde g_{\Lt}$ 
 satisfies $\tilde F^{\l} = M^{\l} F^{\l} k$ for 
 some $z$-independent $\LISU$-valued matrix $M^{\l}$ and a $\Uone$-valued matrix $k$, 
 in particular independent of $\lambda$. For the associated family 
 $\tilde f_{\Lt}^\lambda$ of $\tilde g_{\Lt}$ which is defined  
 by the Sym formula \eqref{eq:SymMin} from the extended frame $\tilde F^{\l}$, 
 we see that $\tilde f_{\Lt}= \tilde f^{\l}_{\Lt}|_{\l =1}$ and $\tilde g_{\Lt}$ 
 are isometric.
 Thus $f_{\Lt}$ and $\tilde f_{\Lt}$ are isometric, 
 and again from \cite[Proposition 1]{Wan} we obtain that  $\tilde f_{\Lt}$ 
 is an entire, complete, spacelike {\sc CMC} graph.
 Let $\tilde f^{\l}$ 
 be the corresponding associated family of minimal surfaces in $\Nil$ which is defined  
 by the Sym formula \eqref{eq:symNil} from
 the extended frame $\tilde F^{\l}$. Then using the argument in 
 Theorem \ref{thm:symcorrespond}, we see that 
 $\tilde f = \tilde f^{\l}|_{\l =1}$ is an entire, complete minimal graph in $\Nil$.
 Note that the Abresch-Rosenberg differential of $\tilde f$ is also $Q \, dz^2$.

 We now apply Lemma \ref{lem:isometry}.  
 If the isometry $f_{\Lt}$ and $\tilde f_{\Lt}$ is of case $(1)$ or $(2)$, 
 then $\tilde f (=\tilde f^{\l}|_{\l =1})$ is congruent to $f(=f^{\l}|_{\l=1})$.
 However, if the isometry of $f_{\Lt}$ and 
 $\tilde f_{\Lt}$ is of case $(3)$, then $\tilde f$ is
 in general non-congruent to $f$. In particular, the case $(3)$ in 
 Lemma \ref{lem:isometry} corresponds to a two-parameter family of boosts in $\Lt$.
 Therefore, for an entire, complete, 
 spacelike CMC surface, there exists a two-parameter family of 
 non-congruent complete minimal graphs in $\Nil$ which have the same 
 Abresch-Rosenberg differential $Q dz^2$. 
\end{proof}

\begin{Remark}
 In \cite{Fer-Mira2}, the two-parameter family of 
 an entire, complete, minimal graph was obtained by the 
 choice of the initial condition for a nonlinear partial differential equation.
 The solution corresponds to $\phi_3$, that is, the $e_3$-component of 
 $f^{-1} f_z$ and the initial condition is the initial value 
 $\phi_3(z_*)$ for some base point $z_*$ in $\mathbb C$ or $\mathbb D$. 
 In our setting, this freedom naturally appears as the 
 two-parameter family of boosts in $\Lt$: 
 As we see from the proof of Theorem \ref{thm:Bern}, two minimal surfaces 
 $f^{\l}$ and $\tilde f^{\l}$ satisfy the relation  \eqref{eq:minirelation}.
 Set 
\be
 M^{\l}|_{\l =1} = 
 \begin{pmatrix}
 \alpha & \beta \\
 \bar{\beta} & \alpha
 \end{pmatrix}
 \in \ISU,\;\;
 (f^{\l})^{-1}f^{\l}_{z}|_{\l=1} = \sum_{k= 1}^3 \phi_k e_k, \;\mbox{and}\;
 (\tilde f^{\l})^{-1}\tilde f^{\l}_{z}|_{\l=1} = \sum_{k= 1}^3 
 \tilde \phi_k e_k,
\ee
 where $\alpha$ is real.
 Then a straightforward computation (using the proof of \cite[Theorem 6.1]{DIK:mini}) 
 shows 
 \begin{equation}
 \tilde \phi_3  = (\alpha^2 + |\beta|^2) \phi_3 + 2 i\Re (\alpha \beta)\phi_1 
 + 2 i \Im (\alpha \beta) \phi_2.
 \end{equation}
 From this expression, it is clear that our two-parameter family of boosts 
 induces a freedom of the initial condition of $\phi_3$ which is naturally 
 parametrized by $\C$.
\end{Remark}
\begin{Corollary}
 The associated family of every entire, complete, minimal graph in $\Nil$ 
 with a given Abresch-Rosenberg differential $Q \,dz^2$ 
 is a family of entire, complete, minimal graphs in $\Nil$
 with the Abresch-Rosenberg differential $\l^{-2} Qdz^2\;(\l \in \mathbb S^1)$.
 Moreover, within a given  associated family, complete minimal graphs 
 have the same support  $h \,(dz)^{1/2} (d \bar z)^{1/2}$.
\end{Corollary}

\begin{proof}
 From the proof of Theorem \ref{thm:Bern}, it is clear that for an  entire minimal graph 
 $f^{\l}|_{\l =1}$ all members of its  associated family of minimal surfaces 
 have  the Abresch-Rosenberg differential $\l^{-2} Q dz^2 \;(\l \in \mathbb S^1)$ 
 and the same support $h\,(dz)^{1/2} (d \bar z)^{1/2}$.
 To prove that the minimal surfaces in the associated family are graphs, we consider the spacelike {\sc CMC} surfaces $f_{\Lt}^{\l}$ given in the 
 proof of Theorem \ref{thm:Bern}. 
Then, since $f_{\Lt}^{\l}|_{\l=1}$ is entire, 
 it is complete \cite[p.~415]{CY} and thus the spacelike {\sc CMC} 
 surfaces $f_{\Lt}^{\l}$ 
 in the associated family are also complete. They are in fact isometric to 
 $f_{\Lt}^{\l}|_{\l=1}$. Note that the complete metric is given by 
 $h^2\, dz d \bar z$, see \eqref{eq:supportmetric}. 
Therefore,  by \cite[Proposition 1]{Wan}, all
 $f_{\Lt}^{\l}$ are entire graphs, and thus the corresponding minimal surfaces
 $f^{\l}$ in $\Nil$ are also 
 entire graphs. 
 The completeness of the associated family follows from Lemma \ref{lm:metricrelation}.
\end{proof}

\begin{Remark}[Canonical examples]
 In \cite{Fer-Mira2}, all entire, complete, minimal vertical graphs 
 are called the {\it canonical examples}. 
\end{Remark}

\appendix
\section{Isometry group of the three-dimensional Heisenberg group}\label{sc:isoNil}
 \subsection{}
 The identity component $\mathrm{Iso}_{\circ}(\Nil)$
 of the isometry group of $\Nil$ is the semi-direct product 
 $\Nil \ltimes \mathrm{SO}_2$. If we identify $\Nil$ with $\C \times \R$ 
 and $\mathrm{SO}_2$ with $\mathrm{U}_1$, respectively, 
 then the action of $\Nil\ltimes \mathrm{SO}_2 \left(\cong 
 (\C \times \R) \ltimes \mathrm{U}_1 \right)$ is given by
\begin{equation}\label{eq:action}
((\alpha=a_1 + i a_2, \;a_3), \; e^{i \theta}) \cdot \left(z=x_1 + i x_2, \;x_3\right)
 =\left(e^{i \theta} z + \alpha, \;x_3+\tfrac{1}{2}\Im ( \bar \alpha e^{i \theta} z)+ a_3\right),
\end{equation}
 where $\theta, a_3, x_3 \in \R$ and $\alpha, z \in \C$.
Here $(x_1, x_2, x_3)$ is a coordinate system of $\Nil$, 
 $\theta$ is a rotation angle and $(a_1, a_2, a_3)$
 is a translation vector.
 The Heisenberg group $\Nil$ is represented by
 $(\Nil\ltimes \mathrm{SO}_2)/\mathrm{SO}_2$ as 
 a naturally reductive homogeneous space. One can see that 
 this homogeneous space is not Riemannian symmetric.
\section{Basic Results}
\label{sc:Preliminaries}
\subsection{Basic notation}
 Let $\Nil$ be the three-dimensional Heisenberg group with the 
 bundle curvature $\tau =1/2$ and let $f : M \to \Nil$ 
 a conformal immersion of a Riemann surface $M$ into $\Nil$.
 Denote  the orthonormal basis 
 of the Lie algebra of $\Nil$ by $\{e_1, e_2, e_3\}$. 
 Then the Maurer-Cartan form $f^{-1} d f$ can be 
 expanded as $f^{-1} d f = (f^{-1} f_z) dz + (f^{-1} f_{\bar z}) d \bar z$
 with $ f^{-1} f_z = \sum_{k =1}^3 \phi_k e_k$
 and $f^{-1} f_{\bar z} =\overline{f^{-1} f_z}  = \sum_{k =1}^3 \bar \phi_k e_k$.
 Here $(z= x + i y)$ are conformal coordinates, $\bar z= x - i y$ is its complex 
 conjugate, and the subscripts $z$ and $\bar z$ denote
 the partial differentiations with respect to $z$ and $\bar z$, respectively. 
 Moreover $\phi_k$ is a complex-valued function and $\bar \phi_k$
 is the its complex conjugate function.
 Since $f$ is a conformal immersion, it is easy to see that $\phi_k (k =1, 2, 3)$
 satisfy $\sum_{k=1}^3 \phi_k^2 =0$ and $\sum_{k =1}^{3}|\phi_k|^2 = \frac{1}{2} 
 e^u\neq 0$. 
 We note that the induced metric of $f$ is given by $ds^2 = e^u dz d \bar z$.
 Then using the {\it generating spinors} $\psi_1$ and $\psi_2$, the first equation 
 can be solved by
 \be
 \phi_1 = (\overline{\psi_2})^2 - \psi_1^2, \;\; 
 \phi_2 = i ((\overline{\psi_2})^2 + \psi_1^2), \;\; 
 \phi_3 = 2 \psi_1\overline{\psi_2}.
 \ee
 Then the condition $\sum_{k=1}^3 |\phi_k|^2 = e^u/2$ is equivalent with 
 $e^u = 4 (|\psi_1|^2 + |\psi_2|^2)^2$.
 Let $N$ be the positively oriented unit normal vector field along $f$ and 
 denote  an unnormalized normal vector field $L$ by $L = e^{u/2} N$. We 
 define the {\it support} $h (dz)^{1/2} (d\bar z)^{1/2}$ by $h = \langle f^{-1} L, e_3\rangle$. 
 Then it is easy to compute $h$ by the generating spinors $\psi_1$ and $\psi_2$: 
 $h = \langle f^{-1}L, e_3\rangle = 2 (|\psi_1|^2-|\psi_2|^2)$.
 Moreover, let $e^{w/2}$ and $Q \,dz^2= 4 B\, dz^2$ be the Dirac potential and
 the Abresch-Rosenberg differential \cite{Abresch-Rosenberg}, given by 
\be
 e^{w/2} = \mathcal U = \mathcal V = -\frac{H}{2}e^{u/2} + \frac{i}{4} h, 
 \hspace{2mm} \mbox{and} \hspace{2mm}  B = \frac{2 H + i}{4} \left(\langle f_{zz}, N \rangle + \frac{\phi_3^2}{2 H + i}\right),
\ee
respectively.
 It is known that the vector of  generating spinors $\tilde \psi = (\psi_1, \psi_2)$ 
 satisfies the so-called ``linear spinor system'' \cite{DIK:mini}:
\begin{equation}\label{eq:linspinor}
 \tilde \psi_z = \tilde \psi \tilde U , \;\;
 \tilde \psi_{\bar z} =  \tilde \psi \tilde V, \;\;
\end{equation}
 where 
 \be 
 \tilde U =
 \begin{pmatrix}
 \frac{1}{2} w_z  + \frac{1}{2} H_{z} e^{-w/2+u/2}&
 - e^{w/2} \\ 
 B e^{-w/2} & 0
 \end{pmatrix}, \;\;
 \tilde V =
 \begin{pmatrix}
 0 & - \bar B e^{-w/2}\\
 e^{w/2} & \frac{1}{2}w_{\bar z}+\frac{1}{2} H_{\bar z}e^{- w/2 + u/2}
 \end{pmatrix}.
\ee
 We note that the second column of the first equation and 
 the first column of the second equation together 
 are the nonlinear Dirac equations, 
 that is, 
 \be 
 \partial_z \psi_2 = - \mathcal U \psi_1, \;\;
 \partial_{\bar z} \psi_1 =  \mathcal V \psi_2,
\ee
 where $\mathcal U = \mathcal V = e^{w/2}$.

\subsection{Flat connections}

 From now on we assume that the unit normal  $f^{-1} N$ is upward, that is, 
 the $e_3$-component of $f^{-1} N$ is positive.
 Since $f^{-1} N$ is upward, there is a stereographic projection $\pi$ of 
 the unit normal $f^{-1} N$ from the south pole to the unit disk in $\mathbb R^2$.
 We denote the map $\pi \circ f^{-1} N$ by $g$ and call $g$ the {\it normal 
 Gauss map}.
 Then it is easy to see that $g$ can be represented by the generating 
 spinors as 
\be
 g = \frac{\psi_2}{\overline{\psi_1}}.
\ee
 We now define the family of Maurer-Cartan forms $\alpha^{\lambda}$ as
 $\alpha^{\l} = \tilde U^{\l} dz + \tilde V^{\l} d\bar z$ with 
 \begin{equation}\label{eq:U-V1lambda}
 \tilde U^{\l} =
 \begin{pmatrix}
 \frac{1}{4} w_z  + \frac{1}{2} H_{z} e^{-w/2+u/2}&
 - \l^{-1}e^{w/2} \\ 
 \l^{-1}B e^{-w/2} &  -\frac{1}{4} w_z
 \end{pmatrix}, \;\;
 \tilde V =
 \begin{pmatrix}
  -\frac{1}{4} w_{\bar z} & - \l \bar B e^{-w/2}\\
 \l e^{w/2} & \frac{1}{4}w_{\bar z}+\frac{1}{2} H_{\bar z}e^{- w/2 + u/2}
 \end{pmatrix}.
 \end{equation}
 Then minimal surfaces in $\Nil$ are characterized in terms of 
 the normal Gauss map as follows.
 \begin{Theorem}[Theorem 5.3 in \cite{DIK:mini}]\label{thm:mincharact}
 Let $f : \D \to \Nil$ be a conformal immersion which is 
 nowhere vertical and $\alpha^{\l}$
 the $1$-form defined in \eqref{eq:U-V1lambda}.
 Moreover, assume that the unit normal $f^{-1} N$ is upward.
 Then the following statements are equivalent{\rm:}
 \begin{enumerate}
 \item $f$ is a minimal surface.
 \item $d + \alpha^{\l}$ is a family of flat connections on $\D \times  \ISU$.
 \item The normal Gauss map $g$ for $f$ is a non-conformal harmonic 
       map into the hyperbolic two-space $\mathbb H^2$.
 \end{enumerate}
 \end{Theorem}
\begin{Definition}\label{def:extended}
 Let $f$ be a minimal surface in $\Nil$ 
 and $F^{\l}$ a $\LISU$-valued solution to the equation
 $(F^{\l})^{-1} d F^{\l} = \alpha^{\lambda}$ such that
 \begin{equation}\label{eq:extmini} 
 F^{\l}|_{\l =1} = 
\frac{1}{\sqrt{|\psi_1|^2-|\psi_2|^2}} 
 \begin{pmatrix}
 \sqrt{i}^{-1} \psi_1 & 
 \sqrt{i}^{-1} \psi_2 \\ 
 \sqrt{i} \;\overline{\psi_2} & 
  \sqrt{i}\; \overline{\psi_1}
 \end{pmatrix}. 
 \end{equation}
 Then $F$ is called an
 \textit{extended frame} of the minimal surface $f$.
\end{Definition}
%
\subsection{Sym-formula}\label{sc:Sym}
First we note that the (multiple of the ) Killing form 
 $\langle A,B\rangle = 4 \tr AB$ induces a Lorentz metric 
 on $\mathfrak{su}_{1,1}$. Thus we regard $\mathfrak{su}_{1,1}$ as 
 the Minkowski 3-space.  The basis
\begin{equation}\label{eq:basis}
 \mathcal{E}_1 = \frac{1}{2} \begin{pmatrix} 0 & i \\ -i &0 \end{pmatrix}, \;\;
 \mathcal{E}_2 = \frac{1}{2} \begin{pmatrix} 0 & -1 \\ -1 & 0 \end{pmatrix}\;\;
 \mbox{and}\;\;\;
 \mathcal{E}_3 = \frac{1}{2} \begin{pmatrix} -i & 0\\ 0 &i \end{pmatrix}
\end{equation}
  is an orthonormal basis of $\isu$ with timelike vector $\mathcal{E}_3$. 

 The timelike vector $\mathcal{E}_3$ generates
 the rotation group $\mathrm{SO}_2$ which acts isometrically on 
 $\mathbb{L}_3$ by rotations around the $x_3$-axis. On the other hand, the
 isometries $\exp(t\mathcal{E}_1)$ and $\exp(t\mathcal{E}_2)$ are called 
 {\it boosts}.

 Now we identify the Lie algebra 
 $\mathfrak{nil}_3$ of $\Nil$ with the 
 Lie algebra $\isu$ as a \textit{real vector space}.  
 A linear isomorphism $\Xi:\mathfrak{su}_{1,1}\to 
 \mathfrak{nil}_3$ is then given by
 \begin{equation}\label{eq:Nilidenti}
\mathfrak{su}_{1,1} \ni 
x_1 \mathcal{E}_1 + x_2 \mathcal{E}_2 + x_3 \mathcal{E}_3
\longmapsto
x_1 e_1 +  x_2 e_2 +  x_3 e_3 \in \mathfrak{nil}_{3}.
\end{equation}
 Note that the linear isomorphism $\Xi$ is not a Lie algebra 
 isomorphism. 
 Next we consider the exponential map 
 $\exp:\mathfrak{nil}_3\to \mathrm{Nil}_3$  
\be
 \exp (x_1 e_1 + x_2 e_2 + x_3 e_3) 
 = e^{x_1} E_{11} + \sum_{i=2}^4 E_{ii} + x_1 E_{2 3}
 + (x_3 +\frac{1}{2}x_1 x_2)E_{2 4} + x_2 E_{34},
\ee
 where $E_{ij}$ is a $4$ by $4$ matrix with the $ij$-entry equal to $1$, 
 and all other entries equal to $0$.
 Here we imbed $\Nil$ into ${\rm GL}_4 \mathbb R$ by $\iota : 
 \Nil \to {\rm GL}_4 \mathbb R$, $\iota(x_1 + x_2+ x_3) 
 = e^{x_1} E_{11} + \sum_{i=2}^4 E_{ii} + x_1 E_{2 3}
 + (x_3 +\frac{1}{2}x_1 x_2)E_{2 4} + x_2 E_{34}$.
 We define a smooth bijection  
 $\Xi_{\rm nil}:\isu \to \Nil$ by $\Xi_{\rm nil}:=\exp \circ \Xi$.
 Under this identification $\Nil=\mathfrak{su}_{1,1}$, 
 $\mathrm{SO}_2=\{\exp(t\,\mathcal{E}_3)\}_{t\in\mathbb{R}}$ 
 acts isometrically on $\Nil$ as rotations around $x_3$-axis.

 In what follows we will take derivatives 
 for functions of $\l$.
 Note that for $\l=e^{i\theta} \in \mathbb S^1$, we have
 $\partial_{\theta}=i \l \partial_{\l}$.
\begin{Theorem}[Theorem 6.1 in \cite{DIK:mini}]\label{thm:Sym}
 For the extended frame $F^{\l}$
 of some minimal surface $f$, define maps $f_{\Lt}^{\l}$ and $N_{\Lt}^{\l}$ 
 respectively by
 \begin{equation}\label{eq:SymMin}
 f_{\Lt}^{\l}=-i \l (\partial_{\l} F^{\l}) (F^{\l})^{-1} 
 -N_{\Lt}^{\l}\;\;
 \mbox{and} \;\;
 N_{\Lt}^{\l}= \frac{i}{2} \ad (F^{\l}) \sigma_3.
 \end{equation}
 Moreover, define a map  $f^{\l}:\mathbb{D}\to \mathrm{Nil}_3$ by
 $f^{\l}:=\Xi_{\mathrm{nil}}\circ \hat{f^{\l}}$ with
\begin{equation}\label{eq:symNil}
 \hat f^{\l} = 
    (f_{\Lt}^{\l})^o -\frac{i}{2} \l (\partial_{\l}  f_{\Lt}^{\l})^d, 
\end{equation}
 where the superscripts ``$o$'' and ``$d$'' denote the off-diagonal and 
 diagonal part, 
 respectively. Then, for each $\l \in \mathbb{S}^1$, the following 
 statements hold$:$ 
\begin{enumerate}
 \item The map $f_{\Lt}^{\l}$ is a spacelike {\sc CMC} surface 
with mean curvature $H=1/2$ in $\Lt$
 and $N_{\Lt}^{\l}$ is the timelike unit normal vector of $f_{\Lt}^{\l}$.
 \item The map $f^{\l}$ is a minimal surface in $\Nil$ and 
 $N_{\Lt}^{\l}$ is the normal Gauss map of $f^{\l}$. 
 In particular, $f^{\l}|_{\l =1}$ gives 
 the original minimal surface $f$ up to translation.
\end{enumerate}
\end{Theorem}

\begin{Remark}\label{rm:AppRem}
\mbox{}
\begin{enumerate}
\item It is known that the Maurer-Cartan form $\alpha^{\l} = \tilde U^{\l} dz +
 \tilde V^{\l} d \bar z$ in \eqref{eq:U-V1lambda} with $H=0$ and $\l =1$ 
 is the Maurer-Cartan form 
 of a spacelike {\sc CMC} surface with mean curvature $H=1/2$, the Hopf differential 
 $Q_{\Lt} dz^2 = - 4 B \,dz^2$ and the metric $h^2 dz d \bar z$, 
 see \cite[Lemma 3.1]{BRS:Min}.\footnote{
 The mean curvature and the Hopf differential 
 for spacelike surfaces in Minkowski space could  be defined differently from 
 the definition in Euclidean space;
 The second fundamental form ${\rm I\!I}$ and the mean curvature $H$  
 could be changed to $-{\rm I\!I}$ and $-H$, respectively, that is,
 in the above case the mean curvature and the Hopf differential would then become 
 $H =-1/2$ and  $Q_{\Lt} \, dz^2= 4 B \, dz^2$, respectively.} 
 Any $\LISU$-valued 
 solution $F^{\l}$ of $(F^{\l})^{-1} d F^{\l} = \alpha^{\l}$ 
 is called the {\it extended frame} of a spacelike {\sc CMC} surface in $\Lt$.
\item  The Hopf differential 
 of the spacelike {\sc CMC} surface $f_{\Lt}^{\l}$ in Theorem 
 \ref{thm:Sym}
 can be computed as $Q_{\Lt}^{\l} dz^2= -4 \l^{-2}B \,dz^2$, where $Q^{\l} \,dz^2 = 4 \l^{-2} B dz^2$ is the Abresch-Rosenberg 
 differential of the minimal surface $f^{\l}$ in $\Nil$.

\item Note that in Theorem \ref{thm:Sym}
 the choice of coordinates is free. 
 We will therefore apply this result to graph coordinates 
 as well as to conformal coordinates without further mentioning.
\end{enumerate}
\end{Remark}

 In the following Corollary, we compute 
 the Abresch-Rosenberg differential $Bdz^2$ for 
 the $1$-parameter family $f^{\l}$ in Theorem \ref{thm:Sym}
 and it implies that the family $f^{\l}$ actually 
 defines the associated family.
\begin{Corollary}\label{coro:associate}
 Let $f^{\l}$ be the family of minimal surfaces in $\Nil$
 defined by \eqref{eq:symNil}.
 Then $f^{\l}$ preserves 
 the mean curvature $(=0)$ and the support. Moreover, the
 Abresch-Rosenberg differential $Q^{\l}dz^2$ for $f^{\l}$
 is given by $Q^{\l} dz^2= 4 \l^{-2}  Bdz^2$, 
 where $Q \,dz^2 = 4 B\, dz^2$ is the Abresch-Rosenberg differential for $f^{\l}|_{\l =1}$. 
 Therefore $\{f^{\l}\}_{\l \in \mathbb S^1}$
 is the associated family of the minimal surface $f^{\l}|_{\l =1}$.
\end{Corollary}
\bibliographystyle{plain}
\def\cprime{$'$}

\end{document}